\documentclass[10pt]{article}
\usepackage{amsfonts} 
\usepackage{pdfsync}
\usepackage{amsthm}
\usepackage{amsmath}
\usepackage{color}
\usepackage[colorlinks=true]{hyperref}

\catcode`@=11 \@addtoreset{equation}{section} \catcode`@=12

\newtheorem{theo}{Theorem}[section]
\newtheorem{prop}[theo]{Proposition}

\newtheorem{lemme}[theo]{Lemma}

\def\R{\mathbb  R}

\title{On radial solutions for a Fully Non Linear degenerate or singular PDE.}
\author{\ \\ C.O. Ndaw }
\date{}
\begin{document}
\maketitle

 \begin{abstract}
  In this paper we consider equations $-| \nabla u |^\alpha F ( D^2 u) = |u|^{p-1} u $ in  an annulus. $F$ is Fully Nonlinear Elliptic, $\alpha$ is some real $> -1$ and $p > 1+ \alpha$. The solutions are intended in the sense of the  definition given by   Birindelli and Demengel \cite{BD1}.
    We prove the existence of non trivial  positive/negative radial solutions. 
    \end{abstract}
    
     \section{Introduction}
     
    In this article we consider positive (or negative) radial solutions of the equation 
    \begin{equation}\label{eqp} \left\{ \begin{array}{lc}
     -| \nabla u |^\alpha F(  D^2 u ) = u^p& {\rm in } \ \Omega\\
     u=0 & {\rm on} \ \partial \Omega \end{array}\right.
     \end{equation}

      where $\Omega$ is an annulus of $\R^N$ , $\alpha >-1$, $F$ is Fully Nonlinear Elliptic, i.e 
      \begin{equation}\label{F0}
      F(0)=0
      \end{equation}
      \begin{equation}\label{FNL}
      \mbox{for all } M, M^\prime  \geq 0 \ {\rm  in}\  {\cal S}_n, \ 
       \lambda tr( M^\prime ) \leq F(  M+ N)-F(M)\leq \Lambda tr(M^\prime )
       \end{equation} 
       where ${\cal S}_N$ is the space of symmetric matrices $(N,N)$
       for some $0<\lambda<\Lambda$. $\alpha $ is supposed to be $> -1$. This assumption enables us to define viscosity solutions in the sense of \cite {BD1}.
        
         The results here enclosed extend the existence of solutions in the case $\alpha = 0$ in \cite{GLP}. 
         
          Before detailing the result and recalling the known results in the case $\alpha = 0$, let us point out that  for general domains, when $\alpha \neq 0$, very few things have been done in the case of a \underline{positive} non linear term $u^p$ on the right hand side, except the case $p=1+\alpha$ which corresponds to the eigenvalue case.  
           It is clear that, if we consider a more general  bounded domain than an annulus, the existence of solutions depends on the conditions 
           $ p> 1+\alpha,\ p = 1+ \alpha \mbox{ or}\ p< 1+ \alpha$. 
           
           In the case $p=1+\alpha$, the existence of positive solutions depends on the first eigenvalue as defined in \cite{BD1}. 
           The other cases are the object of several works in preparation.  
           
           \medskip
           
           Note that for the case $\alpha = 0$ and general domains, the existence's results are mainly obtained by Sirakov and Quaas in \cite{QS2} : 
           
            Using a fixed point argument (more precisely the degree theory of Brouwer), they prove the existence of non trivial positive solutions when, on the right hand side, $u^p$ is replaced by a more general term, which is said however to be sub-linear, in the sense that 
            $$\lim_{ u \rightarrow 0} { f(u) \over u} \geq ( \lambda_1^++\epsilon) ,  \ \forall  u>0,\  f(u)\leq ( \lambda^+_1-\epsilon) u + k$$
             for some fixed $\epsilon >0$ and $k>0$. $\lambda_1^+$ denotes the first  positive eigenvalue. The super-linear case is defined by the conditions $$\limsup_{u\rightarrow 0} {f(u)\over u} <\lambda_1^+ < \liminf _{u\rightarrow +\infty} {f(u)\over u}.$$                In this superlinear case the authors prove the existence of solutions provided that $p \leq p^+ := { n_+ \over n_+-2}$ if $F = { \cal M}_{\lambda,\Lambda}^+$ and $p\leq p^-= { n_-\over n_--2}$ when $F = { \cal M}_{\lambda,\Lambda}^-$,
               where $ n_+ = { \lambda ( N-1)\over \Lambda}+1$, $n_-= { \Lambda ( N-1) \over \lambda} +1$ and ${ \cal M}_{\lambda,\Lambda}^-$ and ${ \cal M}_{\lambda,\Lambda}^-$ are the extremal Pucci's operators (see \cite{CC} for more details). 
              
              Let us point out that this restriction is  due to the method employed which consists in using a blow up procedure and  some Liouville type results which are partly due to Cutri and Leoni, \cite{CL} (see also \cite{FQ}).
              
              However, it is clear that whatever the method employed, there exists a power $p$ beyond which there does not exist, for general domains, non trivial solutions.               This restriction is linked in the variational case to the existence of the critical power. 
              
              Let us recall some facts about the Laplacian case :  
            In that case,   the  equation 
               \begin{equation}\label{lap}\left\{\begin{array}{lc} -\Delta u = |u|^{p-1} u&\ {\rm in } \ \Omega \\
                u = 0 & {\rm on} \ \partial \Omega
                \end{array}\right.
                \end{equation} 
                 does not admit a solution when $N \geq 3,$ $p\geq { N+2 \over N-2}$ and $\Omega$ is star shaped. 
                 
                                  On the contrary there are a lot of non trivial  domains for which there exist non trivial solutions : In particular, the results in \cite{GLP} are expected  since,  in the case of  an annulus, denoted here  $\mathcal{A}$, it is easy to see that the space  ${H_o^1}_{rad} ( \mathcal{A})=\{u\in H^{1}_{0}(\mathcal{A})\ :\ u \mbox{ is radial}\}$ has the property that it is embedded in any $L^p(\mathcal{A})$ and the embedding is compact  when $\mathcal{A}$ is bounded. This implies the existence of solutions   to (\ref{lap}) for any $p>1$,  in that case. 
            
           As we said above,   this is extended to the  non variational case by \cite{GLP}
 where the authors  prove 
 
  \begin{theo}\label{thm1GLP}
   Let $\mathcal{A}_{a, b}$ be the annulus $\mathcal{A}_{a, b} = \{ x \in {\mathbb R}^N, a < |x| < b\}$ with $a>0$. Suppose that  $F$ satisfies (\ref{F0}) and (\ref{FNL}).
   Then  there exist  positive / negative non trivial  radial solutions to 
   $$ -F(D^2 u) = |u|^{p-1} u \ {\rm in} \ \mathcal{A}_{a, b},\ u= 0 \ {\rm on  } \ \partial \mathcal{A}_{a, b}.$$
   \end{theo}
   
   In a second part of the paper they prove the existence of infinitely changing sign solutions, more precisely  they prove : 
    
\begin{theo} Let us suppose that, as in Theorem \ref{thm1GLP},  $F$ satisfies (\ref{F0}) and (\ref{FNL}). Then, for any $k\ \in\ \mathbb{N}$, there exist radial solutions $u^{\pm}_{k}$ of 
\begin{eqnarray*}
-F(D^{2}u)=|u|^{p-1}u \ \ \mbox{in}\ \mathcal{A}_{a,b},\ \ u=0\ \ \mbox{on}\ \partial\mathcal{A}_{a,b}
\end{eqnarray*}
 and finite sequences $(r^{\pm}_{k,j})^{k}_{j=0}\subset[a,\ b]$ such that
\begin{itemize}
\item[(i)]{$a=r^{\pm}_{k,0}<r^{\pm}_{k,1}<\cdots<r^{\pm}_{k,k}=b$;}
\item[(ii$_{+}$)]{$(-1)^{j-1}u^{+}_{k}>0$  in $\mathcal{A}_{r^{+}_{k,j-1}, r^{+}_{k,j}}$ for $j=1, \cdots, k$;}
\item[(ii$_{-}$)]{$(-1)^{j-1}u^{-}_{k}>0$  in $\mathcal{A}_{r^{-}_{k,j-1}, r^{-}_{k,j}}$ for $j=1, \cdots, k$}
\end{itemize} 
where $\mathcal{A}_{r^{+}_{k,j-1}, r^{+}_{k,j}}=D(0,r^{+}_{k,j})\setminus\overline{D(0,r^{+}_{k,j-1})}$ and $\mathcal{A}_{r^{-}_{k,j-1}, r^{-}_{k,j}}=D(0,r^{-}_{k,j})\setminus\overline{D(0,r^{-}_{k,j-1})}$.\\
\end{theo}

     Concerning the case $\alpha \neq 0$, let us summarize some of the results which are the object of several works in preparation : 
     
      Following the arguments in \cite{QS1}, we can prove the existence of solution to the "sub-linear" problem, i.e assuming that for some $\epsilon$ and $k$ positive, 
      
      $$\lim_{ u \rightarrow 0} { f(u) \over u^{1+\alpha} } \geq ( \lambda_1^++\epsilon) ,  \mbox{ and } \forall  u>0,\ f(u)\leq ( \lambda_1^+-\epsilon) u^{1+\alpha}  + k.$$
             In the "super-linear" case, if we argue as in \cite {QS1}, we need a Liouville type result as in \cite{CL},  and a blow up argument. The Liouville result is, ( \cite{D1} )
        \begin{theo}\label{th1}
   Let $n_- = { \Lambda (N-1) \over \lambda} + 1$.   Suppose that $n_{-}>2$. 
  Suppose that 
      $p \leq \tilde p^-:= {n_- (1+ \alpha)-\alpha  \over n_--2}$ then any radial non negative super-solution of 
      $$
       -| \nabla u |^\alpha  F( D^2 u) \geq |u|^{p-1} u\  {\rm in} \ {\mathbb R}^N$$
       
        is identically zero. 
     \end{theo}
Note that the extension to an half space is required to prove the existence of solution provided that $p < \tilde p^+$. It uses the moving plane method and in our setting   this brings new difficulties due to the  failure of the strict comparison principle. 
     
     We now give the main result of this article. Let us recall that if $u$ is radial and ${\cal C}^2$, the eigenvalues of $D^2 u$ are $u^{\prime\prime}$ of order $1$ and ${u^\prime \over r}$ of order $N-1$. As a consequence, if for example $u^\prime \leq 0$ and $u^{ \prime \prime} \geq 0$, 
              $$ {\cal M}^+_{ \lambda, \Lambda} ( D^2 u) = \lambda (N-1){u^\prime \over r} + \Lambda u^{\prime \prime}$$
               while 
                $$ {\cal M}^-_{ \lambda, \Lambda} ( D^2 u) = \Lambda (N-1){u^\prime \over r} + \lambda u^{\prime \prime.}$$
                 By an abuse of notation, in the sequel, we denote in the same manner 
                 $ F( D^2 u)$
 and $F(r,  u^\prime, u^{\prime\prime})$  when $u$ is radial .  
 
  We  want to emphasize that in  the case of  the previous operators, and even in the radial case, we do not necessarily have $u \in { \cal C}^2$. More precisely, when $\alpha >0$ the solutions of an equation 
 $$| \nabla u |^\alpha F( D^2u) = f$$ when $f$ is continuous, are at most ${ \cal C}^{1,{1\over 1+\alpha}}$ ( \cite{BD2}, \cite{IS}), while in the case $\alpha \leq 0$  the radial solutions are ${ \cal C}^2$, (\cite{BD3}). Of course near a point where $u^\prime \neq 0$, the ${ \cal C}^2$ regularity holds for radial solutions. 
       
          We make the following additional assumption, which can probably be dropped. We just need it in the case $\alpha \neq 0$ to prove the local existence and uniqueness near a point where $u^\prime =0$. The restriction below is due to the method employed, and can probably be dropped in a future analysis. 
We then suppose that 
                \begin{equation}\label{kK}
                \begin{array}{lc}
                            \mbox{  There exist} \  k^\pm, K^\pm \mbox{ some positive constants such that }\ 
               \lambda \leq k^\pm \leq K^\pm\leq \Lambda, \ {\rm and} \\
              F(r,  u^\prime, u^{\prime \prime}) = {N-1\over r} (k^+ ( u^\prime)^+- k^- (u^\prime)^-) + K^+ (u^{\prime \prime})^+ -K^- ( u^{\prime \prime})^-.&
              \end{array}
              \end{equation} 
                  
               The main theorem of this paper is
              \begin{theo}\label{exieqp}
       Suppose that $F$ satisfies (\ref{F0}), (\ref{FNL}) and (\ref{kK}). Let $\alpha \in ]-1, +\infty[$, and $p> 1+\alpha$. Let $a>0$ and $b>0$ be given. There exists a positive / negative radial solution  for (\ref{eqp}).
       \end{theo}
              
          In the section 2, we prove the existence of positive solutions  defined in a maximal  intervall, and in the section 3, we show that there exists a maximal intervall which is equal to $\mathcal{A}_{a,b}$.
           
             \section{Local existence result (existence of solutions in a maximal intervall)}

             The idea to prove the existence of a solution is to establish first a local existence's result, which must be treated differently when we are near a point where $u^\prime \neq 0$ or $u^\prime =0$. We begin with the initial conditions $u^\prime ( a) = \gamma$ and $u( a) = 0$.
             
                                      We define then the Cauchy Problem 
                \begin{equation}\label{cauchy}
               \left\{ \begin{array}{lc}
                 -|u^\prime |^\alpha F( r, u^\prime, u^{\prime \prime}) = u^p& \ r>a\\
                  u( a) = 0, u^\prime ( a) = \gamma.&
                  \end{array}\right.
                  \end{equation}
                  
                         Since $u$ must be positif, $\gamma >0$. Suppose that we have obtained the existence on $\R^+$. We can then define a maximal intervall $[ a, \rho ( \gamma)[$ on which $u>0$. We will further prove that  there exists  $\gamma$ so that $\rho( \gamma)  = b$.

         Let us define 
           $G_{k}(x) = k^+  x^+-k^-x^-$, 
                    then  one can write,  near  a point  $r_o\neq 0$ so that $u^\prime ( r_o)\neq 0$ the equation under the form 
           $$ 
           ( x^\prime, y^\prime ) = (f_1(r,x,y), f_2(r,x, y))$$
            where 
            $$f_1(r,x,y) = y,\ \ f_2(r, x, y) = G_{{1\over K}} \left( -{x^p \over | y|^\alpha} -{N-1\over r} G_{k} (y)\right)$$
                    with $x(r_o) = u(r_o), y(r_o) = u^\prime (r_o)$. 
                     With these definitions, since $G_k$ is Lipschitz continuous, it is clear that  so are $f_1, $ and $f_2$ as long as $y\neq 0$. 
                         By the Cauchy Lipschitz Theorem, we have local existence when $u^\prime (r_o) \neq 0$.

We now consider the case $u^\prime (r_o) = 0$ :  In that situation we have immediately $u^{\prime \prime} (r_o) <0$ and then $u^\prime$ and $u^{\prime \prime} \leq 0$  on a neighborhood on the right of $r_o$. 
          
          The following Proposition is the only result where the assumption (\ref{kK}) is needed. 
                      \begin{prop}   
    Suppose that $r_o$ is so that $u^\prime (r_o) = 0$. Let $N_k = { (N-1) k^-\over K^-}$. Let $A>0$ be given and define the map 
       $$ T(u(r)) = u(r_o) -\int_{r_o}^r \left( { 1 \over s^{(N_k-1)(\alpha+1)}} \int_{r_o}^s \frac{\alpha+1}{K^-} u^p (t) t^{(N_k-1)(\alpha+1)} d t\right)^{1 \over 1+ \alpha} ds$$
        Then, there exists $\delta$,  depending only on $u(r_o)$, so that , letting
$$B_{r_o, A,  \delta} = \left\lbrace  u\in {\cal C} ( [r_o, r_o+\delta]), u(r_o) = A,  \ |u(r)-A | \leq { A\over 2}\right\rbrace,$$
     the operator   $ T$ maps $B_{r_o, A,  \delta} $ into itself, and $T$ is a contraction mapping on that set. 
\end{prop} 

\ \\
\begin{proof}  We have that
\begin{eqnarray*}
|T(u)(r) -u(r_o)| &\leq ({3A\over 2} ))^{p\over 1+ \alpha} \frac{\alpha+1}{\alpha+2} \left(\frac{1+\alpha}{((N_k-1)(1+\alpha)+1)K^-}\right)^{\frac{1}{1+\alpha}}|r^{\alpha+2 \over 1+\alpha}-r_o^{\alpha+2 \over 1+\alpha}|\\
&\leq ({3A\over 2} ))^{p\over 1+ \alpha}\left(\frac{1+\alpha}{((N_k-1)(1+\alpha)+1)K^-}\right)^{\frac{1}{1+\alpha}}b ^{\frac{1}{1+\alpha}}|r-r_o|:= c_k
\end{eqnarray*}
           and then it is sufficient to ask 
           $$\delta\leq \left(\frac{((N_k-1)(1+\alpha)+1)K^-}{b(1+\alpha)}\right)^{\frac{1}{\alpha+1}}A^{ 1+\alpha-p\over 1+\alpha} 2^{p-1-\alpha\over 1+\alpha} 3^{-p\over 1+\alpha}  $$
           to ensure that $T$ maps $B_{r_o, A,  \delta} $ into itself. 
           
            Furthermore for $u$ and $v$ in $B( r_o, \delta)$, if $\alpha >0$, let us denote  $d\mu$ the measure $ \frac{\alpha+1}{K^-} t^{(N_k-1)(\alpha+1)} d t\chi_{]r_o, s[}$
            $$X(s) = | u^{ p\over 1+\alpha} |_{ L^{ \alpha+1} ( d\mu)}$$
            
                         and
                         $$  Y(s) = | v^{ p\over 1+\alpha} |_{ L^{ \alpha+1} (  d\mu) }$$
                          Then, by Minkowski inequality, one has 
                          
            \begin{eqnarray*}
           &&|T(u)(r)-T(v)(r)| \leq \int_{r_o}^r  { 1 \over s^{(N_k-1)}} |X(s)-Y(s) | ds \\
          & \leq & \int_{r_o}^r   { 1 \over s^{(N_k-1)}} |u^{ p\over 1+\alpha}-v^{ p\over 1+\alpha}|_{ L^{ \alpha+1} ( ]r_o, s[, d\mu)}ds\\
          &\leq& {p\over 1+\alpha} |u-v|_\infty \left( {3A\over 2}\right)^{p-\alpha-1 \over \alpha+1} \int_{r_o}^r {1\over s^{(N_k-1)}}\left(\int_{r_o}^s d\mu\right)^{1\over 1+\alpha} ds\\
          &\leq &  {p\over 1+\alpha} |u-v|_\infty \left( {3A\over 2}\right)^{p-\alpha-1 \over \alpha+1} \int_{r_o}^r {1\over s^{(N_k-1)}} {s^{N_k-1+{1\over 1+\alpha}}\over( (N_k-1)(1+\alpha)+1)^{1\over 1+\alpha} } ds \\
          &\leq & c_{p,k,\alpha, A} |u-v|_\infty |r^{2+\alpha \over 1+\alpha}-r_o^{2+\alpha \over 1+\alpha}|\\
          &\leq & c^\prime_{p,k, \alpha, A} |u-v|_\infty \delta 
          \end{eqnarray*}

                      where in the last lines we have used $p>1+\alpha$ and the mean value's theorem with $c_{p,k, A, \alpha}$ and $c^{\prime}_{p,k, A, \alpha}$, some constants depending on $p,\ k,\ A,$ and $\alpha$. 
             Then choosing $\delta$ small enough,  so that $$  c^\prime _{p,k,\alpha, A} \delta < {1\over 2},$$ one gets that $T$ is a contraction mapping. 
         
         If $\alpha <0$, we suppose $\delta<1$. One can take $d_{p,k, A, \alpha}$, a constant depending on $p,\ k,\ A,$ and $\alpha$, and
 \begin{eqnarray*}
         d^{\prime}_{p,k, A, \alpha}= { p\cdot d_{p,k, A, \alpha}^{-\alpha \over 1+\alpha}(r_o+1)\over 2K^-(( N_k-1)( \alpha+1)+1)}  \left({3A\over 2}\right)^{p-1} 
\end{eqnarray*} and use the mean value's theorem in the following manner 
      
 \begin{eqnarray*}
           &&|T(u)(r)-T(v)(r)| \leq {1\over 1+\alpha} \left|\int_{r_o}^r  \left( { \alpha+1\over K^- s^{(N_k-1)(\alpha+1)}} \int_{r_o}^s t^{{(N_k-1)(\alpha+1)}}(u^p-v^p)  ( t)  d t\right) ds\right|\times d_{p,k, A, \alpha}^{-\alpha \over 1+\alpha}  \\
           &\leq & p |u-v|_\infty \left({3A\over 2}\right)^{p-1}       {d_{p,k, A, \alpha}^{-\alpha \over 1+\alpha} \over  2K^-((N_k-1)(\alpha+1)+1)}| r^{2}- r_o^{2}| \\
           &\leq & d_{p,k, A, \alpha} |u-v|_\infty  |r-r_o| 
                              \end{eqnarray*}
                              
            
           and the conclusion follows in the same manner.
           \end{proof}
           
            From this result and the remark  done above when $u^\prime (r_o) \neq 0$, we get a maximal intervall $[a, \rho( \gamma)] $ on which $u>0$. 

\section{Existence of positive (or negative) solutions of (\ref{eqp}) in $\mathcal{A}_{a,b}$}  
  
 In the sequel, we will use an abuse by writing $|u^\prime |^\alpha u^{\prime \prime}$ even for points where $u^\prime =0$. : In the case $\alpha <0$ it does not make sense, and in the case $\alpha >0$, $u^{\prime \prime}$ may not exist. More correctly the reader must understand it as  ${d\over dr} ( |u^\prime |^\alpha u^\prime)$, which makes sense, since in any cases $u$ is ${ \cal C}^1$. 
                      
                     In order to limit and simplify some of the computations later, in particular those who do not require condition (\ref{kK}), let us  make a brief analysis  to limit the  cases occurring here (note that  we cannot have both $u^\prime \geq 0$ and $u^{ \prime \prime} \geq 0$):
          
           Case 1 : 
           $u^\prime \geq 0$ and $u^{\prime \prime} \leq 0$, then 
            $$-|u^\prime |^\alpha \left( \lambda u^{\prime \prime} + \Lambda { N-1\over r} u^\prime\right) \leq u^p\leq  -|u^\prime |^\alpha \left( \Lambda u^{\prime \prime} + \lambda {N-1\over r} u^\prime \right);$$
            
            Case 2 : $u^\prime$ and $u^{\prime \prime}$ $\leq 0$,  then 
            $$ -\lambda|u^\prime |^\alpha \left ( u^{\prime \prime}+ {N-1\over r} u^\prime \right) \leq u^p\leq  -\Lambda|u^\prime |^\alpha  \left( u^{\prime \prime}+ {N-1\over r} u^\prime \right); $$
               
               Case 3 :
               $ u^\prime \leq 0$ and $u^{\prime \prime} \geq 0$,
                then 
                
                  $$-|u^\prime |^\alpha \left( \Lambda u^{\prime \prime} + \lambda { N-1\over r} u^\prime\right) \leq u^p\leq 
       -|u^\prime |^\alpha \left( \lambda u^{\prime \prime} + \Lambda {N-1\over r} u^\prime \right).$$
            
             Let us remark in particular that $u^{\prime \prime} <0$ as long as $u^\prime (r) \geq 0$, and then a critical point of $u$ is a strict local maximum. 
   
              Let us recall  that we denote  $\gamma =  u^\prime ( a)$ which is then positive if $u$ is non negative and not identically zero by the strong maximum principle, see \cite{BD1}.  We will then denote $u(\cdot, \gamma)$ a positive radial solution  so that 
              $u^\prime (0, \gamma ) = \gamma$. 
            We begin with the following lemma

           \begin{lemme}\label{lem1}  Let $u$ be a nonnegative  radial solution of $-| \nabla u |^\alpha F( D^2 u) = u^p$, and suppose that $\gamma$  is defined as $u^\prime ( a) = \gamma$. Let $[a,\rho(\gamma))$ be a maximal intervall on which $u$ exists. There exists a unique $\tau(\gamma)\ \in\ (a,\rho(\gamma))$ so that $u^{\prime}(r,\gamma)>0$ for $r<\tau(\gamma)$ and $u^{\prime}(r,\gamma)<0$ for $r>\tau(\gamma).$
\end{lemme}
\begin{proof} Since $u^{\prime}(a)=\gamma>0$, from the equation $u^{\prime\prime}(a)<0$. Therefore, there exists a right neighborhood of $a$ where $u$ satisfies \textbf{Case} 1, and continues to satisfy \textbf{Case} 1 as long as $u$ increases. If we had $u^{\prime}(r)>0$ for all $r\ \in\ [a,\rho(\gamma))$, then $\rho(\gamma)=\infty$ and $u$ would be concave and increasing in $[a,+\infty)$. Thus, there would exist $\lim_{r\rightarrow\infty}u(r)>0$ (possibly $+\infty$), $\lim_{r\rightarrow\infty}u^{\prime}(r)<\gamma$ and, by \textbf{Case} 1,
\begin{eqnarray*}
\lim_{r\rightarrow\infty}\sup u^{\prime\prime}(r)|u^{\prime}|^{\alpha}\leq- \frac{1}{\Lambda}\lim_{r\rightarrow\infty}u^{p}(r)<0,
\end{eqnarray*}
and then
\begin{eqnarray*}
\limsup(|u^\prime |^\alpha u^\prime )^\prime <0
\end{eqnarray*}
which implies that
\begin{eqnarray*}
|u^\prime |^\alpha u^\prime ( +\infty) -|u^\prime |^\alpha u^\prime ( a) = \int_a^{ \infty}  (|u^\prime |^\alpha u^\prime )^\prime  = -\infty,
\end{eqnarray*}
a contradiction to $\lim_{r\rightarrow\infty}u^{\prime}(r)\geq 0 $.\\
Hence, there exists $\tau(\gamma)\ \in\ (a,\rho(\gamma))$ defined as the first zero of $u^{\prime}$. As already observed, any other critical point of $u$ would  be a local strict maximum point for $u$, from which it follows that $u^{\prime}(r)<0$ for all $r\ \in\ (\tau(\gamma),\rho(\gamma))$.

\end{proof}

We now define two energy functionals as in \cite{GLP} : 
\begin{lemme} \label{lem2} Let $n_{-}=\frac{(N-1)\Lambda}{\lambda}+1$ and
\begin{eqnarray*}
E_{1}(r)=r^{(2+\alpha)(n_{-}-1)}\left(\frac{|u^{\prime}|^{\alpha+2}}{\alpha+2}+\frac{u^{p+1}}{(p+1)\Lambda}\right)
\end{eqnarray*}
and
\begin{eqnarray*}
E_{2}(r)=\frac{|u^{\prime}|^{\alpha+2}}{\alpha+2}+\frac{u^{p+1}}{(p+1)\lambda}.
\end{eqnarray*}
Then $E_{1}(r)$ is non decreasing on $[\tau(\gamma),\rho(\gamma)]$, while $E_{2}(r)$ is non increasing on the same intervall.
\end{lemme}
\begin{proof}  By Lemma (\ref{lem1}) we know that $u^{\prime}(r,\gamma)<0$ for $r\geq\tau(\gamma)$. Hence, u satisfies either inequalities in \textbf{Case} 2  or inequalities in \textbf{Case} 3.  Note that, in both the two cases, one has 
\begin{equation}\label{Lambda}
|u^{\prime}|^{\alpha}\left(u^{\prime \prime} + { \Lambda (N-1)\over \lambda r}u^\prime \right) \leq -{u^p \over \Lambda}.
\end{equation} 
By differentiating,  using the fact that $u^\prime <0$  and the  inequality above one gets that
\begin{eqnarray*}
E_{1}^{\prime}(r)&=&(n_{-}-1)r^{(2+\alpha)(n_{-}-1)-1}\left[|u^{\prime}|^{\alpha+2}+\frac{ (2+\alpha)}{\Lambda(p+1)}u^{p+1}(r)\right]\\
&&\ \ +r^{(2+\alpha)(n_{-}-1)}u^{\prime}(r)\left(u^{\prime\prime}(r)|u^{\prime}|^{\alpha}+\frac{u^{p}(r)}{\Lambda}\right)\\
&=& r^{(2+\alpha)(n_{-}-1)}\left[-u^{\prime}\left(-|u^{\prime}|^{\alpha}\left(u^{\prime\prime}(r)+\frac{\Lambda}{\lambda} \frac{N-1}{r}u^{\prime}\right)\right)\right.\\
&&\left. +u^{\prime}\dfrac{u^{p}(r)}{\Lambda}+\dfrac{(2+\alpha)(n_{-}-1)}{\Lambda(p+1)}\dfrac{u^{p+1}}{r}\right]\\
&\geq & r^{(2+\alpha)(n_{-}-1)}\dfrac{(2+\alpha)(n_{-}-1)}{\Lambda(p+1)}\dfrac{u^{p+1}}{r}.\\
\end{eqnarray*}

Analogously, 
\begin{eqnarray*}
E_{2}^{\prime}(r)=u^{\prime}(r)\left(u^{\prime\prime}(r)|u^{\prime}|^{\alpha}+\frac{u^{p}(r)}{\lambda}\right)
\end{eqnarray*}
and by the inequalities in \textbf{Case} 2 and in \textbf{Case} 3, it follows that $u$ satisfies for every $r\geq\tau(\gamma)$
\begin{eqnarray*}
-u^{\prime\prime}|u^{\prime}|^{\alpha}\leq \frac{u^{p}}{\lambda},
\end{eqnarray*}
from which we derive immediately that $E_{2}^{\prime}(r)\leq 0$ for $r\geq\tau(\gamma)$.
\end{proof}
\ \\
The following proposition is an easy consequence of the above lemma. 
\begin{prop} \label{prop1}Let $u$ be a nonnegative  radial solution of $-| \nabla u |^\alpha F( D^2 u) = u^p$ and $[a,\rho(\gamma))$ a maximal intervall on which $u$ exists.
If $\rho(\gamma)=+\infty$, then $\lim_{r\rightarrow+\infty}u(r)=0$.
\end{prop}
            
    \begin{proof}Let us suppose  that $\rho(\gamma)=+\infty$. Then,  since $u$ is  decreasing after $\tau(\gamma)$, it has a finite limit. Let  $c\geq0$ such that $\lim_{r\rightarrow+\infty}u(r,\gamma)=c$. Let us recall that for $r\geq\tau(\gamma)$, $u$ always satisfies (\ref{Lambda}). Furthermore, since $E_2$ is decreasing   on $[\tau( \gamma), +\infty)$,  
$ |u^\prime|^{ \alpha+2} $ is bounded. Hence, $\frac{\Lambda}{\lambda}(N-1)\frac{u^{\prime}}{r}|u^{\prime}|^{\alpha}$ tends to $0$ when $r$ goes to infinity. Finally, by (\ref{Lambda}),
\begin{eqnarray*}
0\leq\lim_{r\rightarrow\infty}\sup u^{\prime\prime}(r)|u^{\prime}|^{\alpha}\leq -\frac{c^{p}}{\Lambda},
\end{eqnarray*} 
which gives $c=0$.
\end{proof}

\begin{prop}\label{proptau}Let $u$ be a nonnegative  radial solution of $-| \nabla u |^\alpha F( D^2 u) = u^p$ and $[a,\rho(\gamma))$ a maximal intervall on which $u$ exists.  Suppose that $\tau$ is defined as in Lemma \ref{lem1} and $p>\alpha+1$. One has
\begin{itemize}
\item[a)]{$\lim_{\gamma\rightarrow+\infty}u(\tau(\gamma),\gamma)=+\infty$,}
\item[b)]{$\lim_{\gamma\rightarrow+\infty}\tau(\gamma)=a$,}
\item[c)]{$\lim_{\gamma\rightarrow0}u(\tau(\gamma),\gamma)=0$,}
\item[d)]{$\lim_{\gamma\rightarrow0}\tau(\gamma)=+\infty$.}
\end{itemize}
\end{prop}
\begin{proof} Let us denote in the following $\tau=\tau(\gamma)$.\\
By  Lemma \ref{lem1}, in the intervall $[a, \tau]$, $u$ is increasing, concave and satisfies  \textbf{Case} 1. Using the left-side inequality in \textbf{Case} 1, we have immediatly that, for $n_{-}$ as in Lemma \ref{lem2}, the following energy function 
\begin{eqnarray*}
\mathcal{E}_{1}(r)=r^{(2+\alpha)(n_{-}-1)}\left(\frac{|u^{\prime}|^{\alpha+2}}{\alpha+2}+\frac{u^{p+1}}{\lambda(p+1)}\right)
\end{eqnarray*}
satisfies $\mathcal{E}_{1}^{\prime}(r)\geq  {(2+\alpha) (N-1)\Lambda \over \lambda} r^{(2+\alpha)(n_{-}-1)}  {u^{p+1}\over \lambda ( p+1)}\geq 0  $ in $[a, \tau]$. Therefore $\mathcal{E}_{1}(\tau)\geq \mathcal{E}_{1}(a)$, and this yields a first lower bound on $u(\tau,\gamma)$:
\begin{eqnarray}
\label{lb1}
u^{p+1}(\tau,\gamma)\geq\dfrac{\lambda(p+1)}{\alpha+2}\left(\frac{a}{\tau}\right)^{(2+\alpha)(n_{-}-1)}\gamma^{\alpha+2}.
\end{eqnarray}
The left-side inequality of \textbf{Case} 1 can also be rewritten as
\begin{eqnarray*}
(u^{\prime})^{\alpha}u^{\prime\prime}+\dfrac{(n_{-}-1)}{r}(u^{\prime})^{\alpha+1}\geq-\frac{u^{p}}{\lambda}.
\end{eqnarray*}
Multiplying both sides by $r^{(n_{-}-1)(\alpha+1)}$ we have 
\begin{eqnarray*}
\left(\dfrac{r^{(n_{-}-1)(\alpha+1)}(u^{\prime})^{\alpha+1}}{\alpha+1}\right)^{\prime}\geq-\dfrac{r^{(n_{-}-1)(\alpha+1)}}{\lambda}u^{p}
\end{eqnarray*}
and integrating from $a$ to $r\ \in\ (a,\tau]$ we obtain
\begin{eqnarray}\label{XYalphapo}
\dfrac{r^{(n_{-}-1)(\alpha+1)}(u^{\prime})^{\alpha+1}}{\alpha+1}(r)&\geq & \dfrac{a^{(n_{-}-1)(\alpha+1)}\gamma^{\alpha+1}}{\alpha+1}-\frac{1}{\lambda}\int_{a}^{r}s^{(n_{-}-1)(\alpha+1)}u^{p}(s)ds\nonumber \\
&\geq &\dfrac{a^{(n_{-}-1)(\alpha+1)}\gamma^{\alpha+1}}{\alpha+1}-\dfrac{r^{(n_{-}-1)(\alpha+1)}u^{p}(r)(r-a)}{\lambda}. 
\end{eqnarray}Let us  divide both sides by $\dfrac{r^{(n_{-}-1)(\alpha+1)}}{\alpha+1}$, and let us  distinguish between the cases $\alpha \geq 0$ and $\alpha <0$ : 

When $\alpha \geq 0$, we use the inequality  $((X-Y)^+)^{\frac{1}{1+\alpha}}\geq X^{\frac{1}{1+\alpha}}-Y^{\frac{1}{1+\alpha}}$ for  $X,Y>0$, with $X^{1+\alpha}  = {\small\dfrac{a^{(n_{-}-1)(\alpha+1)}\gamma^{\alpha+1}}{\alpha+1}}$  and 
 $Y^{1+\alpha} = \dfrac{r^{(n_{-}-1)(\alpha+1)}u^{p}(r)(r-a)}{\lambda}$. Taking the power ${1\over 1+\alpha}$ of  both the left hand side and the positive part of the right hand side of (\ref{XYalphapo}),  one gets \\
\begin{eqnarray*}
u^{\prime}(r)&\geq a^{n_{-}-1} r^{1-n_{-}}\gamma-\left(\dfrac{(\alpha+1)u^{p}(r)(r-a)}{\lambda}\right)^{\frac{1}{\alpha+1}}.
\end{eqnarray*}
When $\alpha <0$, we use, denoting $ Z^{1+\alpha}  = {\small\dfrac{r^{(n_{-}-1)(\alpha+1)}(u^{\prime})^{\alpha+1}}{\alpha+1}}$, 
  $X^{1+\alpha}  = {\small\dfrac{a^{(n_{-}-1)(\alpha+1)}\gamma^{\alpha+1}}{\alpha+1}}$  and 
 $Y^{1+\alpha} = \dfrac{r^{(n_{-}-1)(\alpha+1)}u^{p}(r)(r-a)}{\lambda}$, the convexity inequality   $$X\leq (Z^{1+\alpha}+ Y^{1+\alpha})^{1\over 1+\alpha}  \leq 2^{-\alpha\over 1+\alpha} ( Z+ Y)$$
  and  then 

$$ Z \geq 2^{\alpha\over \alpha+1}  X -Y.$$
Hence  the two cases can be summarized by 
\begin{eqnarray*}
u^{\prime}(r)&\geq  2^{-\alpha^-\over -\alpha^-+1}a^{n_{-}-1} r^{1-n_{-}}\gamma-\left(\dfrac{(\alpha+1)u^{p}(r)(r-a)}{\lambda}\right)^{\frac{1}{\alpha+1}}.
\end{eqnarray*}

Integrating both sides from $a$ to $\tau$, we have the second lower bound

{\small\begin{eqnarray}\label{lb2}
u(\tau,\gamma)\geq  2^{-\alpha^-\over - \alpha^-+1}\dfrac{a^{n_{-}-1}\gamma}{2-n_{-}} \left[ \tau^{2-n_{-}}-a^{2-n_{-}}\right]-(\alpha+1)^{2+ \alpha \over \alpha+1} ( \alpha+2)^{-1} \left(\dfrac{u^{p}(\tau,\gamma)(\tau-a)}{\lambda}\right)^{\frac{1}{\alpha+1}}(\tau-a).
\end{eqnarray}}
with $2-n_{-}<0$.\\

On the other hand, using the right  hand side inequality in  \textbf{Case} 1, we have that the second energy function
\begin{eqnarray*}
\mathcal{E}_{2}(r)=\frac{|u^{\prime}|^{\alpha+2}}{\alpha+2}+\frac{u^{p+1}}{\Lambda(p+1)}
\end{eqnarray*} 
satisfies $\mathcal{E}_{2}^{\prime}(r)\leq -{N-1\over r}|u^\prime |^{\alpha+1} u^\prime  \leq 0$ in $[a,\tau]$. Therefore $\mathcal{E}_{2}(a)\geq \mathcal{E}_{2}(\tau)$, that is the upper bound
\begin{eqnarray}\label{ub1}
u^{p+1}(\tau,\gamma)\leq\dfrac{\Lambda(p+1)}{\alpha+2}\gamma^{\alpha+2}.
\end{eqnarray}
Moreover, again the right hand side inequality in  \textbf{Case} 1 and the fact that $u^{\prime}(r)>0$ for all $r\ \in\ [a,\tau)$ imply that
\begin{eqnarray*}
-\left(\frac{|u^{\prime}(r))|^{\alpha+2}}{\alpha+2}\right)^{\prime}\geq\dfrac{u^{p}(r)u^{\prime}(r)}{\Lambda}.
\end{eqnarray*}
By integrating the above inequality from $r\ \in\ [a,\tau)$ to $\tau$, we obtain
\begin{eqnarray*}
u^{\prime}(r)\geq \left(\dfrac{\alpha+2}{\Lambda(p+1)}[u^{p+1}(\tau)-u^{p+1}(r)]\right)^{\frac{1}{\alpha+2}}.
\end{eqnarray*}
Integrating once again on $[a,\tau]$,  and by  changing  the  variable, we get 
\begin{eqnarray*}
\left(\frac{2+\alpha}{\Lambda(p+1)}\right)^{\frac{1}{\alpha+2}}(\tau-a)\leq \int_{0}^{u(\tau)}{\dfrac{ds}{(u^{p+1}(\tau)-s^{p+1})^{\frac{1}{\alpha+2}}}}=\dfrac{1}{u{(\tau)}^{\frac{p-\alpha-1}{\alpha+2}}}\int_{0}^{1}{\dfrac{dt}{(1-t^{p+1})^{\frac{1}{\alpha+2}}}}.
\end{eqnarray*}
So, letting
\begin{eqnarray*}
c_{p}=\int_{0}^{1}{\dfrac{dt}{(1-t^{p+1})^{\frac{1}{\alpha+2}}}},
\end{eqnarray*}
we get the following second upper bound:
\begin{eqnarray}\label{ub2}
u(\tau,\gamma)\leq \left(\frac{\Lambda(p+1)}{2+\alpha}\right)^{\frac{1}{p-\alpha-1}}\left(\dfrac{c_{p}}{\tau-a}\right)^{\frac{\alpha+2}{p-\alpha-1}}.
\end{eqnarray}

Using (\ref{lb2}) and (\ref{ub2}) we have in both cases $\alpha >0$ and $\alpha \leq 0$,  that
\begin{eqnarray}\label{lastlow}
u(\tau,\gamma)\geq   2^{-\alpha^-\over1-\alpha^-} \dfrac{a^{n_{-}-1} \gamma}{2-n_{-}} \left[ \tau^{2-n_{-}}-a^{2-n_{-}}\right]-C_{\Lambda,\lambda,p,\alpha}(\tau-a)^{-\frac{\alpha+2}{(p-\alpha-1)}}.
\end{eqnarray}
with $C_{\Lambda,\lambda,p,\alpha}$ a constant depending only on $\Lambda,\ \lambda,\ p\ \mbox{and}\ \alpha$.\\
\ \\
Let us now suppose, by contradiction that there exist a positive constant $M$ and a diverging sequence $\gamma_{k}\rightarrow+\infty$ such that
 \begin{eqnarray*}
 u(\tau(\gamma_{k}),\gamma_{k})\leq M.
 \end{eqnarray*}

Using (\ref{lb1}) we have that $\tau(\gamma_{k})\rightarrow+\infty$, which implies, by (\ref{lastlow}), since $n^->2$ 
\begin{eqnarray*}
 u(\tau(\gamma_{k}),\gamma_{k})\geq  2^{-\alpha^-\over1-\alpha^-} \dfrac{a^{n_{-}-1} \gamma_k}{n_{-}-2} a^{2-n_{-}}+o(1) \rightarrow+\infty.
\end{eqnarray*}
That is a contradiction. So \textit{a)} holds. Hence, by (\ref{ub2}),  \textit{b)} follows. Moreover, (\ref{ub1}) immediately gives \textit{c)}. Finally, we observe that, by the concavity, or more  correctly  by the nonincreasing behaviour of $u^\prime$, one has 
\begin{eqnarray*}
u(\tau,\gamma)\leq \gamma(\tau-a)<\gamma\tau.
\end{eqnarray*}
Using the above inequality and (\ref{lb1}), it then follows
\begin{eqnarray*}
\tau(\gamma)^{p+1+(2+\alpha)(n_{-}-1)}\geq\dfrac{\lambda(p+1)}{(\alpha+2)\gamma^{p-\alpha-1}}a^{(2+\alpha)(n_{-}-1)}
\end{eqnarray*}
which proves \textit{d)} since $p>\alpha+1$.
\end{proof}
To prove Theorem \ref{exieqp}, we will need, using the properties of the principal eigenvalue of the Pucci's operators, to  make precise  the behavior of the function $\gamma: \mapsto \rho(\gamma)$. We have the following proposition. 

    \begin{prop}~ \label{gammatoinfty}\\
    Let $u$ be a nonnegative  radial solution of $-| \nabla u |^\alpha F( D^2 u) = u^p$ and $[a,\rho(\gamma))$ a maximal intervall on which $u$ exists.  
\begin{itemize}
\item[i)]{ For every $M>0$, there exists a positive constant $\delta$ depending only on $M,\ a,\ n,\ p,\ \lambda$\mbox{ and } $\Lambda$ such that $$0<\gamma< M\ \Longrightarrow\ a+\delta\leq \rho(\gamma)\leq+\infty.$$ }
\item[ii)]{For $\gamma$ sufficiently large, we have $\rho(\gamma)<+\infty$ and, moreover, $$\lim_{\gamma\longrightarrow+\infty}\rho(\gamma)=a.$$}
\item[iii)]{Finally $$\lim_{\gamma\longrightarrow+\infty}u^{\prime}(\rho(\gamma),\gamma)=-\infty.$$}
\end{itemize}

\end{prop}
\begin{proof}
\textbf{Proof of \textit{i})} : For $0\leq\gamma\leq M$, let us assume $\rho=\rho(\gamma)<+\infty$. Using the uniform ellipticity condition (\ref{FNL}) for $u=u(|x|,\gamma)$ a solution of (\ref{eqp}) we have that
\begin{align*}
-|\nabla u|^{\alpha}\mathcal{M}_{\lambda,\Lambda}^{+}(D^{2}u)&\leq-|\nabla u|^{\alpha}F(D^{2}u)=u^{p-\alpha-1}u^{1+\alpha}\\
&\leq \left(\max_{[a,\rho]}{u}\right)^{p-\alpha-1}u^{1+\alpha}.
\end{align*}
Then the fonction $u$ satisfies, in the annulus $\mathcal{A}_{a,\rho}$, the eigenvalue differential inequality
\begin{eqnarray}\label{ineqmax}
\begin{cases}
-|\nabla u|^{\alpha}\mathcal {M}_{\lambda,\Lambda}^{+}(D^{2}u)\leq \left(\max_{[a,\rho]}{u}\right)^{p-\alpha-1}u^{1+\alpha} \ \ \ \mbox{ in }\ \mathcal{A}_{a,\rho},\\
 \ \ u>0 \ \mbox{ in }\ \mathcal{A}_{a,\rho},\ u=0  \ \mbox{ on }\ \partial\mathcal{A}_{a,\rho}.
\end{cases}
\end{eqnarray}
Let $\lambda^{+}_{1}(-|\nabla u|^{\alpha}\mathcal{M}_{\lambda,\Lambda}^{+},\mathcal{A}_{a,\rho})$ be  the principal eigenvalue of the operator $-|\nabla u|^{\alpha}\mathcal{M}_{\lambda,\Lambda}^{+}$ in $\mathcal{A}_{a,\rho}$ associated with positive eigenfunctions.
Using the characterization of the sign of the principal eigenvalues in terms of the validity of the maximum principle, see \cite{BD1}, \cite{BEQ}, \cite{BNV} and \cite{EFQ}, the above inequality gives that 
\begin{eqnarray*}
\left(\max_{[a,\rho]}{u}\right)^{p-\alpha-1}\geq \lambda^{+}_{1}(-|\nabla u|^{\alpha}\mathcal{M}_{\lambda,\Lambda}^{+},\mathcal{A}_{a,\rho}). 
\end{eqnarray*}
Now, by  Lemma \ref{lem1}, we have the existence of a unique $\tau(\gamma)\in [a,\rho(\gamma)]$ such that $\max_{[a,\rho]}{u}=u(\tau(\gamma),\gamma)$. Using the bound (\ref{ub1}) on $u(\tau(\gamma),\gamma)$ and by obvious  scaling properties with respect to the domain of $\lambda^{+}_{1}$, we then obtain
\begin{eqnarray*}
\lambda^{+}_{1}(-|\nabla u|^{\alpha}\mathcal{M}_{\lambda,\Lambda}^{+},\mathcal{A}_{1,\rho /a})=a^{\alpha+2}\lambda^{+}_{1}(-|\nabla u|^{\alpha}\mathcal{M}_{\lambda,\Lambda}^{+},\mathcal{A}_{a,\rho})\leq a^{\alpha+2}\left(\dfrac{\Lambda(p+1)}{\alpha+2}\gamma^{\alpha+2}\right)^{ p-\alpha-1\over p+1}.
\end{eqnarray*}
We observe that, as in the case $\alpha = 0$, we have $\lim_{ | D| \rightarrow 0} \lambda_1^+(D) = +\infty$. This can be obtained as a  consequence of the Aleksandrof Backelman Pucci's estimate, due to \cite{DFQ}, Theorem 1.2 for the case $\alpha <0$, and to \cite{I}, Theorem 1 for the case $\alpha \geq 0$.  Indeed, applying this estimate to an eigenfunction $u$, with $|u|_\infty = 1$ here, we get 
$$ 1= \sup_{x\in \overline{ \Omega} } u\leq C\lambda_1^+(D)  ( \int_D |u|^{1+\alpha})^{1\over 1+\alpha} = C \lambda_1^+(D)  |D|^{1\over 1+\alpha}.$$
From this one gets 
$ \lambda_1^+(D) \geq {1 \over C |D|^{1\over 1+\alpha}}$.

Then, the above inequality shows that, for $\gamma$ bounded, the ratio $\rho(\gamma) /a$ is greater than 1 and bounded, and this is equivalent to statement (i).\\
\textbf{Proof of \textit{ii})} :  By Proposition \ref{proptau} (b), the statement is equivalent to 
\begin{eqnarray*}
\lim_{\gamma\rightarrow+\infty}{\rho(\gamma) \over \tau(\gamma)}=1.
\end{eqnarray*}
We will argue by contradiction as in \cite{GLP}. Let assume that there exist a sequence $\gamma_{k}\rightarrow+\infty$ and a constant $\delta$, such that, setting $\tau_{k}=\tau(\gamma_{k})$ and $\rho_{k}=\rho(\gamma_{k})$, with possibly $\rho_{k}=+\infty$, one has
\begin{eqnarray*}
\rho_{k}>(1+\delta)\tau_{k}\ \ \ \ \ \ \ \mbox{for all } k\geq 1.
\end{eqnarray*}
This means that $u_{k}(r)=u(r,\gamma_{k})$ is strictly positive in the intervall $[\tau_{k},(1+\delta)\tau_{k}]$. Then, using the uniform ellipticity condition (\ref{FNL}), it follows that $u_{k}$, as a function of $x$, satisfies in the annulus $\mathcal{A}_{\tau_{k},r}$ the eigenvalue differential inequality
\begin{eqnarray*}
u_{k}>0,\ \ \ \ -|\nabla u|^{\alpha}\mathcal{M}^{-}_{\lambda,\Lambda}(D^{2}u_{k}(x))\geq \left(\min_{[\tau_{k},r]}{u_{k}}\right)^{p-\alpha-1}u_{k}^{1+\alpha}(r), \ \ x \in \mathcal{A}_{\tau_{k},r}
\end{eqnarray*}
for every $r\in (\tau_{k},(1+\delta)\tau_{k}]$.\\
Denoting by $\lambda^{+}_{1}(-|\nabla u|^{\alpha}\mathcal{M}^{-}_{\lambda,\Lambda},\mathcal{A}_{\tau_{k},r})$ the principal eigenvalue of $-|\nabla u|^{\alpha}\mathcal{M}^{-}_{\lambda,\Lambda}$ in the domain $\mathcal{A}_{\tau_{k},r}$ associated with positive eigenfunctions, it follows that
\begin{eqnarray*}
\left(\min_{[\tau_{k},r]}{u_{k}}\right)^{p-\alpha-1}\leq \lambda^{+}_{1}(-|\nabla u|^{\alpha}\mathcal{M}^{-}_{\lambda,\Lambda},\mathcal{A}_{\tau_{k},r}),\ \ \ \mbox{ for all } r \in (\tau_{k},(1+\delta)\tau_{k}].
\end{eqnarray*}
By Lemma \ref{lem1}, $u_{k}(r)$ is monotone decreasing for $\tau_{k}\leq r< \rho(\gamma_{k})$, so that $\min_{[\tau_{k},r]}{u_{k}}=u_{k}(r)$. Moreover, by the homogeneity and monotonicity properties of $\lambda^{+}_{1}$ with respect to the domain, one has, for every $r\in [(1+\delta /2)\tau_{k}, (1+\delta)\tau_{k}]$, 
\begin{eqnarray*}
\lambda^{+}_{1}(-|\nabla u|^{\alpha}\mathcal{M}^{-}_{\lambda,\Lambda},\mathcal{A}_{\tau_{k},r})={1 \over \tau^{2+\alpha}_{k}}\lambda^{+}_{1}(-|\nabla u|^{\alpha}\mathcal{M}^{-}_{\lambda,\Lambda},\mathcal{A}_{1,{r \over \tau_{k}}})\leq {1 \over a^{2+\alpha}}\lambda^{+}_{1}(-|\nabla u|^{\alpha}\mathcal{M}^{-}_{\lambda,\Lambda},\mathcal{A}_{1,1+\delta /2}).
\end{eqnarray*}
It is clear that that $\lambda^{+}_{1}(-|\nabla u|^{\alpha}\mathcal{M}^{-}_{\lambda,\Lambda},\mathcal{A}_{1,1+\delta /2})$ is a positive number depending only on $n,\ \lambda, \ \Lambda,\ \alpha$ and $\delta$ and denote it by $C_{\delta}$. We then obtain the uniform estimate
\begin{eqnarray*}
u_{k}(r)^{p-\alpha-1}\leq {C_{\delta}\over a^{2+\alpha} } \ \ \ \ \ \mbox{ for all } r\in [(1+\delta /2)\tau_{k},(1+\delta)\tau_{k}].\ \ \ \ \ \ \ \ \ \ \ \  \ \ \ \ \ \ \ (\star)
\end{eqnarray*}
On the other hand, by Lemma \ref{lem2} we also have 
\begin{eqnarray*}
E_{1}(\tau_{k})\leq E_{1}(r) \ \ \ \ \ \ \ \ \mbox{ for all } r \in [\tau_{k},\rho(\tau_{k})),
\end{eqnarray*} 
and therefore for all $r\in [(1+\delta /2)\tau_{k},(1+\delta)\tau_{k}]$,
\begin{eqnarray*}
|u^{\prime}_{k}(r)|^{(2+\alpha)}\geq {1 \over \Lambda(p+1)}\left[{u_{k}(\tau_{k})^{p+1} \over (1+\delta)^{(2+\alpha)(n_{-}-1)}}-\left({C_{\delta} \over a^{2+\alpha}}\right)^{p+1 \over p-\alpha-1 }\right]:=M_{\delta}(k).
\end{eqnarray*}
By  Proposition  \ref{proptau} (a), one also has 
\begin{eqnarray*}
\lim_{k\rightarrow+\infty}{M_{\delta}(k)}=+\infty.
\end{eqnarray*}
Recall that we have $u^{\prime}_{k}(r)\leq 0$ for all $r$ in $[\tau_{k},\rho(\gamma_{k})]$. We then deduce that
\begin{eqnarray*}
-u^{\prime}_{k}(r)\geq (M_{\delta}(k))^{1\over 2+\alpha} ,\ \ \ \ \ \ \mbox{ for all } r\in [(1+\delta /2)\tau_{k},(1+\delta)\tau_{k}].
\end{eqnarray*}
By integration, this implies 
\begin{eqnarray*}
u_{k}((1+\delta /2)\tau_{k})\geq -\int_{(1+\delta /2)\tau_{k}}^{(1+\delta)\tau_{k}}{u^{\prime}_{k}(r) dr}\geq \delta \tau_{k}{\left( {{M_{\delta}(k) }}\right)^{1\over 2+\alpha}\over 2} \geq {\delta a\over 2}\left({{M_{\delta}(k) }}\right)^{1\over 2+\alpha} ,
\end{eqnarray*}
which contradicts $(\star)$ when $k\rightarrow+\infty$.\\
\textbf{Proof of \textit{iii})} : Let $\gamma$ be  large enough so that $\rho(\gamma)<+\infty$. By Lemma \ref{lem2} ($E_{1}$ increasing) we have $E_{1}(\rho(\gamma))\geq E_{1}(\tau(\gamma))$. And since (\ref{lb1}) holds, we have
\begin{equation*}
\rho(\gamma)^{(2+\alpha)(n_{-}-1)}\dfrac{|u^{\prime}(\rho(\gamma),\gamma)|^{2+\alpha}}{\alpha+2}\geq\tau(\gamma)^{(2+\alpha)(n_{-}-1)}\dfrac{u(\tau(\gamma),\gamma)^{p+1}}{\Lambda(p+1)}\geq\dfrac{\lambda}{\Lambda(2+\alpha)}a^{(2+\alpha)(n_{-}-1)}\gamma^{2+\alpha}.
\end{equation*}
Considering again the Lemma \ref{lem1} and using Hopf Boundary's Lemma, we have that $u^{\prime}(\rho(\gamma),\gamma)<0$, and we obtain
\begin{equation*}
u^{\prime}(\rho(\gamma),\gamma)\leq- \left(\dfrac{\lambda}{\Lambda}\right)^{\frac{1}{2+\alpha}}\left(\dfrac{a}{\rho(\gamma)}\right)^{n_{-}-1}\gamma.
\end{equation*}
Finally, since from  (ii) $\lim_{\gamma\longrightarrow+\infty}\rho(\gamma)=a$, we conclude that
$$\lim_{\gamma\longrightarrow+\infty}u^{\prime}(\rho(\gamma),\gamma)=-\infty.$$.

\end{proof}
   
    \textbf{Proof of Theorem \ref{exieqp}.} We first prove the existence of a positive solution for (\ref{eqp}) and deduce the existence of a negative one.
    
     For every $\gamma>0$, let $u(\rho(\gamma),\gamma)$ be the maximal positive solution of the Cauchy problem (\ref{cauchy}), defined on the maximal interval $[a,\rho(\gamma))$, and satisfying $u(\rho(\gamma),\gamma)=0$ if $\rho(\gamma)<+\infty$. We know that $u(a,\gamma)=0$, $u$ is positive,  radial and satisfy
\begin{equation*}
|\nabla u|^{\alpha}F(D^{2}u)+u^{p}=0 \ \ \ \ \ \ 
\end{equation*}
$\mbox{in}\ \ \mathcal{A}_{a,\rho(\gamma)}$. So we only need to prove the existence of a $\gamma>0$ such that $\rho(\gamma)=b$.
\\
We follow \cite{GLP} and define the set 
\begin{align*}
D=\{\gamma\geq0,\ \ \rho(\gamma)< +\infty\}.
\end{align*}
Comming back to the problem (\ref{cauchy}), we have that $D$ is an open set and $\rho: \gamma\mapsto \rho(\gamma)$ is continuous on $D$. From Proposition \ref{gammatoinfty}, we have that for $\gamma$ large enough, $\rho(\gamma)<+\infty$. This implies that $D$ is nonempty and contains a neighborhood of $+\infty$.
Let $(\gamma^{*},+\infty)$ be the unbounded connected component of $D$, with $\gamma^{*}\geq 0$.\\
If $\gamma^{*}= 0$, since $\rho(\gamma)>\tau(\gamma)$, from Proposition \ref{proptau} \textit{d)}, we obtain that
\begin{equation*}
\lim_{\gamma\searrow 0}{\rho(\gamma)}=+\infty.
\end{equation*}
If $\gamma^{*}> 0$, using again the continuous dependence on the initial data for problem (\ref{cauchy}), we have 
\begin{equation*}
\lim_{\gamma\searrow \gamma^{*}}{\rho(\gamma)}=+\infty.
\end{equation*}
By  Proposition \ref{gammatoinfty}, we also deduce that  $\rho$ is onto from $[\gamma^\star, +\infty)$ in $[a, +\infty)$. Hence if $b\in [ a, +\infty)$, there exists $\gamma$ in $[ \gamma^\star,+\infty)$ so that $\rho( \gamma ) = b$. So there exists a positive solution of (\ref{eqp}) in $\mathcal{A}_{a,b}$.\\
Let us now set 
\begin{equation*}
G(x,M)=-F(x,M).
\end{equation*}
$G$ also satisfies (\ref{F0})  (\ref{FNL}) and (\ref{kK}). We have then the existence of $v\in \mathcal{A}_{a,b}$ such that 
\begin{eqnarray*}
\begin{cases}
-|\nabla v|^{\alpha}G(D^{2}v)=|v|^{p-1}v\mbox{\ \ in } \  \mathcal{A}_{a,b}\\
v>0, \ \ \ v=0\mbox{ \ on} \  \partial\mathcal{A}_{a,b}.
\end{cases}
\end{eqnarray*}
And taking $u=-v$, we have that $u$ is a negative radial solution of 
\begin{eqnarray*}
\begin{cases}
-|\nabla u|^{\alpha}F(D^{2}u)=|u|^{p-1}u\mbox{\ \ in } \  \mathcal{A}_{a,b}\\
u<0, \ \ \ u=0\mbox{ \ on} \  \partial\mathcal{A}_{a,b}.
\end{cases}
\end{eqnarray*}
Thus Theorem  \ref{exieqp} is proved. $\hfill\square$

    \end{document}